\documentclass[a4paper,reqno,oneside]{amsart}

\usepackage{ifthen}
\usepackage[utf8]{inputenc}
\usepackage[english]{babel}
\usepackage{amsmath,amsfonts,amssymb,amsthm}
\usepackage{graphicx}
\usepackage{xcolor}
\usepackage{mathtools}
\usepackage{hyperref}
\usepackage{enumerate}

\newtheorem{theorem}{Theorem}

\newtheorem{corollary}[theorem]{Corollary}

\theoremstyle{definition}

\theoremstyle{remark}

\newcommand{\ignore}[1]{} 

\newcommand{\eps}{\varepsilon}
\renewcommand{\phi}{\varphi}

\def\canrightarrow{\rightarrow\hspace*{-0.48cm}\hbox to 0.48cm{\raisebox{-0.5ex}{{\tiny can}}}}
\def\ordrightarrow{\rightarrow\hspace*{-0.48cm}\hbox to 0.48cm{\raisebox{-0.5ex}{{\tiny ord}}}}
\def\partrightarrow{\longrightarrow\hspace*{-0.6cm}\hbox to 0.48cm{\raisebox{0.9ex}{{\tiny part}}}\hspace*{0.2cm}}

\def\nDeltaeq{\mathrel{\ooalign{
  $\triangleq$\cr
  \hidewidth\raise.2ex\hbox{\scalebox{0.65}{$\not\mkern18mu$}}\cr}}}

\let\eps=\varepsilon

\newcounter{thmlistctr}

\usepackage{stmaryrd}
\def\ex{\text{ex}}

\title{On the number of graphs without large cliques}

\begin{document}

\begin{center}

\LARGE On the number of graphs without large cliques
\vspace{7mm}

\Large{
\begin{center}
\begin{tabular}{ccccc}
Frank Mousset &\quad& Rajko Nenadov &\quad& Angelika Steger
\end{tabular}
\end{center}
}
\vspace{7mm}

\large
  Institute of Theoretical Computer Science \\
  ETH Z\"urich, 8092 Z\"urich, Switzerland \\
  {\small\{{\tt moussetf|rnenadov|steger}\}{\tt @inf.ethz.ch}}
\vspace{5mm}

\small

\begin{minipage}{0.9\linewidth}
\textsc{Abstract.}

In 1976 Erd\H{o}s, Kleitman and Rothschild determined the number of graphs without a clique of size $\ell$.
In this note we extend their result to the case of forbidden cliques of increasing size. 
More precisely we prove that for $\ell_n \le (\log n)^{1/4}/2$ there are $$2^{(1-1/(\ell_n-1))n^2/2+o(n^2/\ell_n)}$$ $K_{\ell_n}$-free graphs of order $n$.
Our proof is based on the recent hypergraph container theorems of Saxton, Thomason and Balogh, Morris, Samotij, in combination with a theorem of Lov\'asz and Simonovits.
\end{minipage}
\vspace{5mm}
\end{center}

\section{Introduction}

Let $F$ be an arbitrary graph. A graph $G$ is called $F$-free if $G$ does not contain $F$ as a (weak) subgraph. 
Let $f_n(F)$ denote the number of (labeled) $F$-free graphs on $n$ vertices. As every subgraph of an $F$-free graph is also $F$-free, we trivially have $f_n(F)\ge 2^{\ex(n,F)}$, where $\ex(n,F)$ denotes the maximum number of edges of an $F$-free graph on $n$ vertices. It is well known \cite{t41,es46,ES66} that
$$\ex(n,F)= \left(1-\frac1{\chi(F)-1}\right)\frac{n^2}2 + o(n^2).$$ 
Erd\H{o}s, Kleitman and Rothschild~\cite{ekr76} showed that in the case of cliques, i.e., for $F=K_\ell$, this lower bound actually provides the correct order of magnitude. 
Erd\H{o}s, Frankl and R\"odl~\cite{EFR} later showed that a similar result holds for all graphs $F$ of chromatic number $\chi(F)\ge 3$:
\begin{equation}\label{eq:fixedl}
{f_n(F)} = 2^{(1+o(1))\ex(n,F)}.
\end{equation}\label{eq:struc}
Note that these results just provide the asymptotics of $\log_2(f_n(F))$. Extending an earlier result from~\cite{ekr76} for triangles, Kolaitis, Pr\"omel and Rothschild~\cite{KPR} determined the typical structure of $K_{\ell}$-free graphs by showing that almost all of them are $(\ell-1)$-colorable. Thus,
\begin{equation}
f_n(K_\ell) = (1+o(1))\cdot\text{col}_n(\ell-1),
\end{equation}
where $\text{col}_n(\ell)$ denotes the number of (labeled) $\ell$-colorable graphs on $n$ vertices. An asymptotic for $\text{col}_n(\ell)$ is given in~\cite{PS95}. For additional results and further pointers to the literature see e.g.~\cite{ABBM11, bbs04, BS11, PS92}.

All of the above  results consider the case of a {\em fixed} forbidden graph $F$. Much less is known if the size of the forbidden graph $F$ increases with the size of the host graph $G$. The study of such 
situations was started only recently by Bollobás and Nikiforov~\cite{BollobasNikiforov10}. They showed that for any sequence $(F_n)$ of graphs with $v(F_n)=o(\log{n})$ one has
\begin{equation}\label{eq:bollobas-nikiforov}
  \log_2{f_n(F_n)} =\left(1-\frac{1}{\chi(F_n)-1}\right)\binom{n}{2}+o(n^2)\text.
\end{equation}
It is interesting to note that the proof of
\eqref{eq:bollobas-nikiforov} completely avoids the use of the regularity lemma, a common tool for attacking this kind of questions. Indeed, because of the tower-type dependence of the size of an $\varepsilon$-regular partition on the parameter $\varepsilon$ (see \cite{Gowers97}), it seems hard to adapt the regularity-based proof of \eqref{eq:fixedl} to the case of forbidden subgraphs of non-constant size. 
Furthermore, observe that \eqref{eq:bollobas-nikiforov} is only non-trivial if the chromatic number $\chi(F_n)$ is bounded. In particular, it does not determine $\log_2 f_n(K_{\ell_n})$
for an increasing sequence $\ell_n$ of positive integers, because the term $o(n^2)$ {swallows} the lower-order term $\binom{n}{2}/(\ell_n-1)$.

The aim of this paper is to provide the first non-trivial result for forbidden cliques of increasing size:
\begin{theorem}
  \label{thm:1}
  Let $(\ell_n)_{n\in\mathbb{N}}$ be a sequence of positive integers such that for every $n\in \mathbb N$,
  we have $3\leq \ell_n\le(\log{n})^{1/4}/2$. Then
  $$ \log_2{f_n(K_{\ell_n})}= \left(1-\frac{1}{\ell_n-1}\right)\binom{n}{2}+o(n^2/\ell_n).$$ 
\end{theorem}
Our proof is based on the recent powerful hypergraph container results of Balogh, Morris, Samotij~\cite{bms12}
and Saxton and Thomason~\cite{SaxtonThomason12}. 

The upper bound on $\ell_n$ in our theorem is an artifact of our proof. We have no reason to believe that this bound is tight. In fact, it is not unconceivable that the statement from Theorem~\ref{thm:1} holds up  to 
the size of a maximal clique in the random graph $G_{n,1/2}$ which is known to be $(2+o(1))\log_2 n$.
We leave this question to future research.

Note also that, similarly to the result of Erd\H{o}s, Kleitman and Rothschild~\cite{ekr76}, our theorem just provides the asymptotics of the logarithm of $f_n(K_{\ell_n})$. However, our paper has already stimulated further research,
and very recently a
 structural result in the spirit of Kolaitis, Pr\"omel and Rothschild
has been established by Balogh et al.~\cite{BLS}.

\section{Hypergraph Containers}

In the proof of Theorem \ref{thm:1}, we make use of the hypergraph container theorem proved independently by Saxton and Thomason~\cite{SaxtonThomason12} and Balogh, Morris and Samotij~\cite{bms12}. Before we state this theorem, we introduce some notation.

Let $H$ be an $r$-uniform hypergraph with the average degree $d$. Then for every $\sigma \subseteq V(H)$, we define
the \emph{co-degree}
$$ d(\sigma) = |\{e \in E(H) : \sigma \subseteq e\}|. $$
Moreover, for every $j \in [r]$, we define the $j$-th \emph{maximum co-degree}
$$ \Delta_j = \max{\{d(\sigma) : \sigma\subseteq V(H)\text{ and }|\sigma|=j\}}. $$
Finally, for any $p\in(0,1)$, we define the function
$$ \Delta(H,p) = 2^{\binom{r}{2}-1}\sum_{j=2}^{r}2^{-\binom{j-1}{2}}\frac{\Delta_j}{d p^{j-1}}. $$
We will use the following version of the general hypergraph container
theorem. \begin{theorem}[{Saxton-Thomason~\cite{SaxtonThomason12}}]
  \label{thm:container}
  There exists a positive integer $c$ such that the following holds for all
  positive integers $r$ and $N$. Let $H$ be an
  $r$-uniform hypergraph of order $N$. Let
  $0\leq p\leq 1/(cr^{2r})$ and $0< \varepsilon< 1$ be such that $\Delta(H,p)\leq \varepsilon/(cr^r)$. Then there exists a collection $\mathcal{C}\subseteq \mathcal{P}(V(H))$ such that
  \begin{enumerate}[(i)]
  \item every independent set in $H$ is contained in some $ C\in \mathcal{C}$,
  \item for all $C\in\mathcal{C}$, we have $e(H[C]) \leq \varepsilon e(H)$, and
  \item the number $|\mathcal{C}|$ of containers satisfies
   $$ \log{|\mathcal{C}|} \leq cr^{3r}(1+\log(1/\eps))Np\log(1/p). $$
  \end{enumerate}
\end{theorem}
Theorem \ref{thm:container} is an easy consequence of Theorem 5.3 in the paper of Saxton and
Thomason~\cite{SaxtonThomason12}; it is derived exactly as Corollary 2.7 (also therein), the only difference being that
we are precise about the dependence on the edge size $r$. We also mention that
our notation deviates slightly from that used in~\cite{SaxtonThomason12}: we write $p$ instead of $\tau$ and we use $\Delta(H,p)$
as an upper bound for the function $\delta(H,\tau)$ used by Saxton and Thomason. Finally,
let us just note without further explanation that
Theorem~\ref{thm:container} is much weaker than
the general container theorem, although it is sufficient for the purposes of this note (and is simpler to state and apply).

As a corollary of Theorem~\ref{thm:container} we prove the following version tailored for a collection of $K_{\ell}$-free graphs, with $\ell$ being a function of $n$.

\begin{corollary}
  \label{cor:container}
  For every constant $\delta>0$
  and sequence $(\ell_n)_{n\in\mathbb N}$ such that $3\le \ell_n\le(\log{n})^{1/4}/2$, the following holds
  for all large enough $n\in\mathbb N$: there
  exists a collection $\mathcal{G}$ of graphs of order $n$ such that
  \begin{enumerate}[(i)]
  \item every $K_{\ell}$-free graph of order $n$ is a subgraph of some $G \in \mathcal{G}$,
  \item every $G\in\mathcal{G}$ contains at most $\delta\binom{n}{\ell_n}/e^{\ell_n}$ copies of $K_{\ell_n}$, and
  \item the number $|\mathcal{G}|$ of graphs in the collection satisfies
    $$ \log{|\mathcal{G}|} \leq \delta n^2/\ell_n. $$
  \end{enumerate}
\end{corollary}
\begin{proof}
  Let us assume that $n$ is large enough and write $\ell:=\ell_n$.

  Let $H$ be a hypergraph defined as follows: the vertex set of $H$ is the edge set of $K_n$, and the edges of $H$ are the edge sets of subgraphs of $K_{n}$ isomorphic to $K_{\ell}$. Observe that the graph $H$ is an $\binom{{\ell}}{2}$-uniform hypergraph of order $\binom{n}{2}$ with $e(H)=\binom{n}{{\ell}}$.
  With some foresight, we would like to apply Theorem~\ref{thm:container} with 
  \begin{equation}\varepsilon = \delta e^{-{\ell}} \quad \text{and} \quad p = n^{- (\log{\ell})/(2\ell^2)}\label{eq:p}\end{equation}
  to the hypergraph $H$. We first verify that this is indeed possible, that is, that
  \begin{equation} \label{eq:delta_c}
    \Delta(H, p) \leq \delta \left(c\cdot e^{\ell}\binom{\ell}{2}^{\binom{\ell}{2}}\right)^{-1}
  \end{equation}
  and
  \begin{equation} \label{eq:p_c}
    p \leq \left( c\cdot \binom{\ell}{2}^{2\binom{\ell}{2}}\right)^{-1}\text,
  \end{equation}
  for every positive integer constant $c$.
  
  Let us start with the values $\Delta_j$ for $H$. Consider some $\sigma\subseteq V(H)$ with $|\sigma|=j$, where $1\leq j \leq \binom{\ell}{2}$. We can view $\sigma$ as a subgraph of $K_n$ with $v(\sigma) = \left| \bigcup \{e:e\in \sigma\} \right|$ vertices and $|\sigma|=j$ edges. The co-degree $d(\sigma)$ is then simply the number of ways in which we can extend this graph to a copy of $K_{\ell}$ in $K_{n}$. If $v(\sigma) > \ell$ then clearly $d(\sigma) = 0$, and otherwise
  $$ d(\sigma) = \binom{n-v(\sigma)}{{\ell}-v(\sigma)}\leq n^{{\ell}-v(\sigma)}. $$
  Note that $j\leq \binom{v(\sigma)}{2}$ implies that
  \[ v(\sigma)\geq \frac{1+\sqrt{1+8j}}{2}> \frac12+\sqrt{2j}\text,\]
  giving the bound
  \[\Delta_j\leq n^{\ell-1/2-\sqrt{2j}}\text.\]
  On the other hand, using that
  that $\ell\leq (\log{n})^{1/4}/2$ and that $n$ is sufficiently large, the average degree $d$ of $H$ is
  \[ d = \binom{n-2}{{\ell}-2} \geq \left(\frac{n}{{\ell}}\right)^{{\ell}-2}\ge n^{\ell-1.9}\text, \]
  so for $2\leq j \leq \binom{\ell}{2}$, we have
  \[ \frac{\Delta_j}{dp^{j-1}}\leq n^{1.4-\sqrt{2j}+(j-1)(\log{\ell})/(2\ell^2)}\text. \]
  Using the fact that $\log(\ell)/\ell\leq 1/e$ holds for all $\ell>0$, we have, for every $2\leq j\leq
  \binom{\ell}{2}$, that
  \[\sqrt{2j}-\frac{(j-1)(\log{\ell})}{2\ell^2} \ge
  \sqrt{2j}-\frac{(\log{\ell})\sqrt{j}}{2\sqrt{2}\ell} \ge \sqrt{2j}- \frac{\sqrt{j}}{2e\sqrt{2}}
  \geq 2-\frac{1}{2e}\text,\]
  whence, for sufficiently large $n$,
  \[ \frac{\Delta_j}{dp^{j-1}}\leq n^{1.4-2+1/(2e)}  \leq n^{-1/4}\text.\]
    Then, using $\ell \le (\log{n})^{1/4}/2$, for large enough $n$, we get
  \[ \Delta(H, p) \leq e^{\ell^4} \cdot \sum_{j = 2}^{\binom{\ell}{2}} 2^{- \binom{j - 1}{2}} \cdot
  n^{-1/4}
  \leq e^{\ell^4} n^{-1/4}
  \leq \delta/(ce^{\ell^4})\text,\]
  which easily implies the desired bound \eqref{eq:delta_c} on $\Delta(H, p)$. On the other hand, again
  using $\ell\le (\log{n})^{1/4}/2$, we have
  \begin{equation}\label{eq:upper_p}
    p = n^{-(\log{\ell})/(2\ell^2)}\leq 1/(c\ell^{4\ell^2})
  \end{equation}
  for all large enough $n$, so $p$ satisfies \eqref{eq:p_c}.
  Therefore, we can apply Theorem~\ref{thm:container} with parameters $\varepsilon$ and $p$.
  
  We now turn to the construction of the family $\mathcal{G}$. Let $\mathcal{C}$ be a collection of subsets of $V(H)$ given by Theorem~\ref{thm:container}. We show that the family of graphs
  $$\mathcal{G} = \{ ([n], C) : C \in \mathcal{C}\}$$
  satisfies the claim.
  
  Suppose that $I$ is some $K_{\ell}$-free graph on the vertex set $[n]$. Then its edge set $E(I)$ is an independent set in $H$, and thus there exists $C \in \mathcal{C}$ such that $E(I) \subseteq C$. Therefore there exists $G \in \mathcal{G}$ such that $I$ is a subgraph of $G$, and the property \textit{(i)} holds. Furthermore, since $e(H[C]) \leq \varepsilon e(H)$ for each $C \in \mathcal{C}$, it follows that the number of copies of $K_{\ell}$ in each $G \in \mathcal{G}$ is also bounded by $\binom{n}{\ell} / e^{\ell}$, satisfying property \textit{(ii)}.  It remains to show that $\log{|\mathcal{C}|} = o(n^2 / \ell)$, which then implies property \textit{(iii)}. Straightforward calculation yields that for large enough $n$, we have
  \begin{align*}
  \log{|\mathcal{C}|} &\leq c\binom{{\ell}}{2}^{3\binom{{\ell}}{2}}(1+\ell-\log\delta)\binom{n}{2}p\log(1/p) \\
  &\stackrel{\text{\eqref{eq:upper_p}}}{\leq} \ell^{3\ell^2}(1+\ell-\log\delta)n^2 \left( \ell^{-4\ell^2}\cdot \log (c\ell^{7\ell^2}) \right) \\
  &\leq \delta n^2/\ell,
  \end{align*}
  where in the second line, we used \eqref{eq:upper_p} together with the fact that $p\log (1/p)$ is monotonically decreasing. This finishes the proof of the corollary.
\end{proof}

The requirement that $\ell_n\leq (\log{n})^{1/4}/2$ cannot be significantly improved upon with the
same method.
Indeed, the requirement that $\Delta(H,p)=o(1)$ implies that
$2^{\binom{\binom{\ell_n}{2}}{2}}\Delta_2/d= o(1)$, which, since $\Delta_2/p=n^{-1+o(1)}$,
implies $\ell_n=O((\log{n})^{1/4})$. We also note that the proof shows that, in fact, we have
$\log{|\mathcal G|}= n^2e^{-\Omega(\ell_n^2\log{\ell_n})}$, which is much stronger than
the bound $\log{|\mathcal G|}=o(n^2/\ell_n)$ that we need for the proof of Theorem~\ref{thm:1}.

\section{Proof of Theorem \ref{thm:1}}

Let us start with the easy part -- proving the lower bound.
Consider the $(\ell_n-1)$-partite \emph{Turán graph}. That is, let $T$ be the complete $(\ell_n-1)$-partite graph of order $n$ whose partite sets have size either $\lceil n/(\ell_n-1) \rceil$ or $\lfloor n/(\ell_n-1) \rfloor$.
 Clearly, $T$ is a $K_{\ell_n}$-free graph, as is every subgraph of $T$. As there are at least
$$ 2^{e(T)} \geq 2^{(\frac{n}{{\ell_n}-1}-1)^2\binom{{\ell_n}-1}{2}} \geq
2^{(1-\frac{1}{{\ell_n}-1})\binom{n}{2}+o(n^2/{\ell_n})} $$
subgraphs of $T$, the lower bound on the number of $K_{\ell_n}$-free graphs of order $n$ follows.

Now we turn to proving the upper bound.
We show that for every $\delta > 0$ and large enough $n$, we have
$$ \log{f_n(K_{\ell_n})} \leq \left(1-\frac{1-\delta}{{\ell_n}-1}\right)\binom{n}{2}+\delta n^2/{\ell_n}. $$

We use the following Theorem of Lovász and Simonovits.
\begin{theorem}[{Lovász-Simonovits~\cite[Theorem 1]{ls83}}]
  \label{thm:supersat}
  Let $n$ and $\ell$ be positive integers. Then every graph of order $n$ with at least
  \[ \left(1-\frac{1}{t}\right)\frac{n^2}{2} \]
  edges contains at least
  $\left(\frac{n}{t}\right)^{\ell}\binom{t}{\ell}$
  copies of $K_{\ell}$.
\end{theorem}

Using Theorem~\ref{thm:supersat} together with Corollary~\ref{cor:container}, we can now finish the proof of Theorem~\ref{thm:1} as follows.
Fix some $\delta>0$ and assume that $n$ is large enough. Write $\ell:=\ell_n$ and apply Corollary~ \ref{cor:container} for $\delta := \delta ^{1/\delta}$.
We deduce that there exists a collection $\mathcal{G}$ of at most $2^{\delta n^2/\ell}$ graphs of
order $n$ such that each contains at most $\delta^{1/\delta}\binom{n}{\ell}/e^{\ell}$ copies of
$K_{\ell}$ and every $K_{\ell}$-free graph of order $n$ is a subgraph of some
$G \in \mathcal{G}$.

By Theorem \ref{thm:supersat}, if a graph $G$ of order $n$ has at least 
$$ \left(1-\frac{1-\delta}{{\ell}-1}\right)\frac{n^2}{2} $$
edges, then the number of copies of $K_{\ell}$ in $G$ is at least
$$ k(\ell) := \left(\frac{n(1-\delta)}{{\ell}-1}\right)^{\ell}\binom{({\ell}-1)/(1-\delta)}{{\ell}}.$$
If $\ell\geq 1/\delta$, then $(\ell-1)/(1-\delta)\geq \ell$ and we can use the bounds $ (\frac{a}b)^b\le \binom{a}b\le (\frac{ea}b)^b$ to obtain
\[k(\ell)\geq \frac{n^{\ell}}{{\ell}^{\ell}} > \binom{n}{{\ell}}/e^{{\ell}}\text.\]
For $\ell<1/\delta$, we use the definition of the (generalized) binomial coefficient to deduce that, in this case,
\[k(\ell)\geq \frac{n^{\ell}}{\ell!}\cdot \prod_{i=1}^{\ell-1}\left(1-\frac{i(1-\delta)}{\ell-1}\right)
> \delta^{1/\delta}\binom{n}{\ell}\text.\]
Since every $G\in\mathcal{G}$ has at most $\delta^{1/\delta}\binom{n}{\ell}/e^{\ell}$ copies of $K_{\ell}$,
we deduce that every $G\in\mathcal G$ has fewer than
$$ \left(1-\frac{1-\delta}{{\ell}-1}\right)\frac{n^2}{2} $$
edges. We can now count the number of $K_{\ell}$-free graphs by counting the number of subgraphs of order $n$ of the graphs in $\mathcal{G}$,
$$ f_n(K_{\ell}) \leq |\mathcal{G}| \cdot 2^{(1-\frac{1-\delta}{{\ell}-1})\frac{n^2}{2}}
= 2^{(1-\frac{1-\delta}{{\ell}-1})\frac{n^2}{2}+\delta n^2/{\ell}},$$
completing the proof of the theorem.

\bibliographystyle{siam}
\bibliography{refs}

\end{document}